\documentclass[12pt]{article}
\usepackage{amsmath}
\usepackage{amssymb}
\usepackage{amsthm}
\usepackage{verbatim}
\usepackage{mathrsfs}

\newcommand{\D}{{\mathcal D}}

\newcommand{\R}[0]{\mathbb R}

\newcommand{\Ds}[0]{\mathcal D}
\newcommand{\N}[0]{\mathbb N}

\newtheorem{Th}{Theorem}[section]
\newtheorem{Lemma}{Lemma}[section]
\newtheorem{Prop}[Lemma]{Proposition}
\newtheorem{Coro}[Th]{Corollary}

\begin{document}

\title{On the well-posedness of the inviscid SQG equation}
\author{H. Inci}

\maketitle

\begin{abstract}
In this paper we consider the inviscid SQG equation on the Sobolev spaces $H^s(\R^2)$, $s > 2$. Using a geometric approach we show that for any $T > 0$ the corresponding solution map, $\theta(0) \mapsto \theta(T)$, is nowhere locally uniformly continuous.
\end{abstract}

\section{Introduction}\label{section_introduction}

The initial value problem for the inviscid SQG equation is given by
\begin{equation}\label{sqg}
\partial_t \theta + (u \cdot \nabla) \theta = 0,\quad \theta(0)=\theta_0
\end{equation}
where $\theta:\R \times \R^2 \to \R$ is a scalar function and $u$ is the velocity of the flow given by
\[
 u=\left(\begin{array}{c} u_1 \\ u_2 \end{array}\right) = \left( \begin{array}{c} -\mathcal R_2 \theta \\ \mathcal R_1 \theta \end{array}\right)
\]
Here we denote by $\mathcal R_1, \mathcal R_2$ the Riesz transforms
\[
 \mathcal R_k = \partial_k (-\Delta)^{-1/2},\quad k=1,2
\]
Our main interest in this equation is because of its similarities with the incompressible Euler equation -- take a look at \cite{cmt1,cmt2} for this relation and the physics of \eqref{sqg}.\\
Because of the special structure of $u$, the flow is incompressible. One can prove local well-posedness of \eqref{sqg} in $H^s(\R^2)$ for $s>2$ using the same techniques as for the incompressible Euler equation -- see e.g. \cite{majda}. We will establish this using a geometric approach.

\begin{Th}\label{lwp}
The inviscid SQG equation is locally well-posed in the Sobolev spaces $H^s(\R^2), s > 2$.
\end{Th}

In the following we denote for $s$ fixed and $T > 0$, the set $U_T \subseteq H^s(\R^2)$ to be the set of those initial values $\theta_0 \in H^s(\R^2)$ for which the solution of \eqref{sqg} exists longer than time $T$. With this our main result reads as

\begin{Th}\label{nonuniform}
The solution map $U_T \to H^s(\R^2), \theta_0 \mapsto \theta(T)$ is nowhere locally uniformly continuous.
\end{Th}

The same result was established in \cite{euler} for the incompressible Euler equation and in \cite{b_family} for the Holm-Staley $b$-family of equations. To establish Theorem \ref{nonuniform} we will use the same techniques as in \cite{euler,b_family}. The idea is to use some sort of ''gliding hump''. If we denote by $\varphi$ the flow of $u$, i.e.
\[
 \partial_t \varphi = u \circ \varphi,\quad \varphi(0)=\mbox{id}
\]
then we have
\[
 \theta(T) \circ \varphi(T) = \theta_0 \quad \mbox{or} \quad \theta(T)=\theta_0 \circ \varphi(T)^{-1}
\]
which is the key to produce the ''gliding hump''. To accomplish that one needs to control $\varphi$. Here the geometric formulation comes into play, which is nothing other than the formulation of \eqref{sqg} in the Lagrangian variable $\varphi$.

\section{Geometric formulation}\label{section_geometricformulation}

In this section we describe the equation \eqref{sqg} in a geometric way. The principle is not new -- see e.g. \cite{arnold,ebin_marsden}. It works quite for a lot of equations. For the incompressible Euler equation \cite{euler}, for the Holm-Staley $b$-family of equations \cite{b_family}, for the Burgers equation and so on. All these equations can be written in the form
\[
 \partial_t u + (u \cdot \nabla) u = F(u)
\]
where on the right hand side there is no loss of regularity. Now consider the flow map of $u$ which we denote by $\varphi$. Using $\varphi_t = u \circ \varphi$ and taking the derivative of this expression one gets
\[
 \varphi_{tt} = u_t \circ \varphi + [(u \cdot \nabla) u] \circ \varphi
\]
Hence $\varphi_{tt}=F(\varphi_t \circ \varphi^{-1}) \circ \varphi$. This is a second order equation in $\varphi$. Let us establish this for equation \eqref{sqg}. Applying $-\mathcal R_2$ to \eqref{sqg} we get
\[
 \partial_t u_1 + (u \cdot \nabla) u_1 = \mathcal R_2((u \cdot \nabla)\theta)- (u \cdot \nabla) \mathcal R_2 \theta = [u \cdot \nabla, -\mathcal R_2]\theta
\]
where $[A,B]=AB-BA$ denotes the commutator of the operators $A,B$. Similarly applying $\mathcal R_1$ to \eqref{sqg} we get
\[
 \partial_t u_2 + (u \cdot \nabla) u_2 = [u \cdot \nabla,\mathcal R_1] \theta
\]
Replacing $\theta=\mathcal R_2 u_1 - \mathcal R_1 u_2$ and recasting both equations we get
\begin{equation}\label{alternative}
\partial_t u + (u \cdot \nabla) u= \left( \begin{array}{c} \left[u \cdot \nabla,-\mathcal R_2\right](\mathcal R_2 u_1 - \mathcal R_1 u_2) \\ \left[u \cdot \nabla,\mathcal R_1\right](\mathcal R_2 u_1 - \mathcal R_1 u_2) \end{array}\right)=:B(u,u)
\end{equation}
with $B$ the quadratic expression in $u$ on the right. Introducing the variables $(\varphi,v)$ where $\varphi$ is the flow map of $u$ and $v$ is $\varphi_t$ we can rewrite \eqref{alternative} as an equation on $\D^s(\R^2) \times H^s(\R^2;\R^2)$
\begin{equation}\label{lagrangian}
 \partial_t \left( \begin{array}{c} \varphi \\ v \end{array}\right) = \left( \begin{array}{c} v \\ B(v \circ \varphi^{-1},v \circ \varphi^{-1}) \circ \varphi \end{array} \right)
\end{equation}
where the function space $\D^s(\R^2)$ is defined for $s > 2$ as
\[
 \D^s(\R^2)=\{ \varphi:\R^2 \to \R^2 \;|\; \varphi - \mbox{id} \in H^s(\R^2;\R^2), \det(d_x \varphi) > 0 \quad \forall x \in \R^2 \}
\]
This space consists of diffeomorphisms of $\R^2$ which are perturbations of the identity map. It is connected and has a differential structure by considering it as an open subset of $H^s(\R^2;\R^2)$. Furthermore it is a topological group under composition of maps. For details one can consult \cite{composition}. Note that the quadratic nature of $B$ makes \eqref{lagrangian} to a geodesic equation on $\D^s(\R^2)$. One of the main difficulties is to prove the regularity of the equation \eqref{lagrangian}. We have

\begin{Prop}\label{prop_analytic}
Let $s > 2$. Then the map
\[
 \D^s(\R^2) \times H^s(\R^2;\R^2) \to H^s(\R^2;\R^2),\quad (\varphi,v) \mapsto B(v \circ \varphi^{-1},v \circ \varphi^{-1}) \circ \varphi
\]
is real analytic.
\end{Prop}

The proof is in the Appendix. An immediate consequence of this proposition is that we get by Picard-Lindel\"of local solutions to \eqref{lagrangian} which are unique. In the following we show that \eqref{alternative} is an equivalent formulation of \eqref{sqg} and in the next section the equivalence of \eqref{lagrangian} and \eqref{alternative}. But first we make for \eqref{sqg} the notion of solution precise. For $s > 2$ we say that $\theta \in C([0,T];H^s(\R^2))$ is a solution to \eqref{sqg} on $[0,T]$ if we have
\[
 \theta(t) = \theta_0 + \int_0^t -(u \cdot \nabla) \theta \;ds \quad \forall \; 0 \leq t \leq T
\]
as an equality in $H^{s-1}$. Note that $u=(-\mathcal R_2 \theta,\mathcal R_1 \theta) \in C([0,T];H^s(\R^2;\R^2))$ and so the integral lies in $C([0,T];H^{s-1}(\R^2;\R^2))$ by the Banach algebra property of $H^{s-1}$.

\begin{Lemma}\label{lemma_div}
Let $s > 2$ and $T > 0$. For $u \in C([0,T];H^s(\R^2;\R^2))$ a solution to \eqref{alternative} with divergence-free initial value, i.e. $\operatorname{div} u(0) = 0$ we have
\[
 \operatorname{div} u(t) = 0 \quad \forall \; 0 \leq t \leq T
\]
\end{Lemma}

\begin{proof}
We denote $u_0=u(0)$. By \eqref{alternative} we have
\[
 u(t) = u_0 + \int_0^t B(u,u) - (u \cdot \nabla) u \; ds \quad \forall \; 0 \leq t \leq T
\]
Let $\Phi=\mathcal R_1 u_1 + \mathcal R_2 u_2$. Note that as $u_0$ is divergence free, we have $\Phi(0)=0$. Applying $\mathcal R_1$ to the first component and $\mathcal R_2$ to the second component one gets
\begin{eqnarray*}
\partial_t \Phi &=& \mathcal R_1([u \cdot \nabla,-\mathcal R_2](\mathcal R_2 u_1 - \mathcal R_1 u_2)) + \mathcal R_2([u \cdot \nabla,\mathcal R_1](\mathcal R_2 u_1- \mathcal R_1 u_2))\\
&& - \mathcal R_1((u \cdot \nabla)u_1)-\mathcal R_2((u \cdot \nabla) u_2)
\end{eqnarray*}
We consider the terms seperately. We have
\begin{eqnarray*}
\mathcal R_1([u \cdot \nabla,-\mathcal R_2](\mathcal R_2 u_1 - \mathcal R_1 u_2)) &=& -\mathcal R_1 u_1 \mathcal R_2^2 \partial_1 u_1 + \mathcal R_1 \mathcal R_2 u_1 \mathcal R_2 \partial_1 u_1 \\
&& -\mathcal R_1 u_2 \mathcal R_2^2 \partial_2 u_1 + \mathcal R_1 \mathcal R_2 u_2 \mathcal R_2 \partial_2 u_1 \\
&& +\mathcal R_1 u_1 \mathcal R_1 \mathcal R_2 \partial_1 u_2 - \mathcal R_1 \mathcal R_2 u_1 \mathcal R_1 \partial_1 u_2 \\
&& +\mathcal R_1 u_2 \mathcal R_1 \mathcal R_2 \partial_2 u_2 - \mathcal R_1 \mathcal R_2 u_2 \mathcal R_1 \partial_2 u_2
\end{eqnarray*}
Similarly we have
\begin{eqnarray*}
\mathcal R_2([u \cdot \nabla,\mathcal R_1](\mathcal R_2 u_1 - \mathcal R_1 u_2)) &=& \mathcal R_2 u_1 \mathcal R_1 \mathcal R_2 \partial_1 u_1 - \mathcal R_1 \mathcal R_2 u_1 \mathcal R_2 \partial_1 u_1 \\
&& + \mathcal R_2 u_2 \mathcal R_1 \mathcal R_2 \partial_2 u_1 - \mathcal R_1 \mathcal R_2 u_2 \mathcal R_2 \partial_2 u_1 \\
&& - \mathcal R_2 u_1 \mathcal R_1^2 \partial_1 u_2 + \mathcal R_1 \mathcal R_2 u_1 \mathcal R_1 \partial_1 u_2 \\
&& - \mathcal R_2 u_2 \mathcal R_1^2 \partial_2 u_2 + \mathcal R_1 \mathcal R_2 u_2 \mathcal R_1 \partial_2 u_2
\end{eqnarray*}
Finally we have
\begin{eqnarray*}
-\mathcal R_1((u \cdot \nabla)u_1)-\mathcal R_2((u \cdot \nabla)u_2) &=& -\mathcal R_1 u_1 \partial_1 u_1 - \mathcal R_1 u_2 \partial_2 u_1 \\
&& -\mathcal R_2 u_1 \partial_1 u_2 - \mathcal R_2 u_2 \partial_2 u_2
\end{eqnarray*}
Using the identity $-\mathcal R_1^2-\mathcal R_2^2 = \mbox{id}$ we rewrite this as
\begin{multline*}
-\mathcal R_1((u \cdot \nabla)u_1)-\mathcal R_2((u \cdot \nabla)u_2) = \\
-\mathcal R_1 u_1( -\mathcal R_1^2-\mathcal R_2^2) \partial_1 u_1 - \mathcal R_1 u_2 (-\mathcal R_1^2-\mathcal R_2^2) \partial_2 u_1 \\
-\mathcal R_2 u_1(-\mathcal R_1^2-\mathcal R_2^2) \partial_1 u_2 - \mathcal R_2 u_2 (-\mathcal R_1^2-\mathcal R_2^2) \partial_2 u_2
\end{multline*}
Adding up we find
\begin{eqnarray*}
\partial_t \Phi &=& \mathcal R_1 u_1 \mathcal R_1^2 \partial_1 u_1 + \mathcal R_1 u_2 \mathcal R_1^2 \partial_2 u_1 + \mathcal R_1 u_1 \mathcal R_1 \mathcal R_2 \partial_1 u_2 + \mathcal R_1 u_2 \mathcal R_1 \mathcal R_2 \partial_2 u_2\\
&& +\mathcal R_2 u_1 \mathcal R_1 \mathcal R_2 \partial_1 u_1 + \mathcal R_2 u_2 \mathcal R_1 \mathcal R_2 \partial_2 u_1 + \mathcal R_2 u_1 \mathcal R_2^2 \partial_1 u_2 + \mathcal R_2 u_2 \mathcal R_2^2 \partial_2 u_2 \\
&=& (\mathcal R_1 u_1 \mathcal R_1 \partial_1 + \mathcal R_1 u_2 \mathcal R_1 \partial_2 + \mathcal R_2 u_1 \mathcal R_2 \partial_1 + \mathcal R_2 u_2 \mathcal R_2 \partial_2)(\mathcal R_1 u_1 + \mathcal R_2 u_2) \\
&=& (\mathcal R_1 u_1 \mathcal R_1 \partial_1 + \mathcal R_1 u_2 \mathcal R_1 \partial_2 + \mathcal R_2 u_1 \mathcal R_2 \partial_1 + \mathcal R_2 u_2 \mathcal R_2 \partial_2) \Phi
\end{eqnarray*}
Hence we have
\[
 \partial_t \frac{1}{2} \langle \Phi,\Phi \rangle_{L^2}=\langle (\mathcal R_1 u_1 \mathcal R_1 \partial_1 + \mathcal R_1 u_2 \mathcal R_1 \partial_2 + \mathcal R_2 u_1 \mathcal R_2 \partial_1 + \mathcal R_2 u_2 \mathcal R_2 \partial_2) \Phi,\Phi \rangle_{L^2}
\]
Using integration by parts the righthand side is
\begin{eqnarray*}
&& \langle (\mathcal R_1 u_1 \mathcal R_1 \partial_1 + \mathcal R_1 u_2 \mathcal R_1 \partial_2 + \mathcal R_2 u_1 \mathcal R_2 \partial_1 + \mathcal R_2 u_2 \mathcal R_2 \partial_2) \Phi,\Phi \rangle_{L^2}= \\
&& - \langle \Phi, \partial_1(\mathcal R_1 u_1 \mathcal R_1 \Phi) + \partial_2(\mathcal R_1 u_2 \mathcal R_1 \Phi) + \partial_1 (\mathcal R_2 u_1 \mathcal R_2 \Phi) + \partial_2 (\mathcal R_2 u_2 \mathcal R_2 \Phi)\rangle_{L^2} = \\
&& - \langle \Phi, \mathcal R_1 \partial_1 u_1 \mathcal R_1 \Phi + \mathcal R_1 \partial_2 u_2 \mathcal R_1 \Phi + \mathcal R_2 u_2 \mathcal R_2 \partial_1 u_1 \mathcal R_2 \Phi + \mathcal R_2 \partial_2 u_2 \mathcal R_2 \Phi \rangle_{L^2} \\
&& -\langle \Phi,(\mathcal R_1 u_1 \mathcal R_1 \partial_1 + \mathcal R_1 u_2 \mathcal R_1 \partial_2 + \mathcal R_2 u_1 \mathcal R_2 \partial_1 + \mathcal R_2 u_2 \mathcal R_2 \partial_2) \Phi \rangle_{L^2}
\end{eqnarray*}
Thus we have
\[
 \partial_t \frac{1}{2} \langle \Phi,\Phi \rangle_{L^2} = -\frac{1}{2} \langle \Phi, \mathcal R_1 \partial_1 u_1 \mathcal R_1 \Phi + \mathcal R_1 \partial_2 u_2 \mathcal R_1 \Phi + \mathcal R_2 u_2 \mathcal R_2 \partial_1 u_1 \mathcal R_2 \Phi + \mathcal R_2 \partial_2 u_2 \mathcal R_2 \Phi \rangle_{L^2} 
\]
Using the Sobolev imbedding $H^{s-1} \hookrightarrow L^\infty$ we have on $[0,T]$ therefore the estimate
\[
 \partial_t ||\Phi||_{L^2}^2 \leq C ||\Phi||_{L^2}^2
\]
As $\Phi(0)=0$ we get by Gronwalls inequality $\Phi \equiv 0$ showing that $\operatorname{div} u=0$.
\end{proof}

A consequence of Lemma \ref{lemma_div} is

\begin{Prop}\label{prop_equivalence}
Let $s > 2$ and $T > 0$. For $\theta_0 \in H^s(\R^2)$ let $\theta \in C([0,T];H^s(\R^2))$ be a solution to \eqref{sqg}. Then $u=(-\mathcal R_2 \theta,\mathcal R_1 \theta)$ is a solution to \eqref{alternative} on $[0,T]$. On the other hand if $u \in C([0,T];H^s(\R^2;\R^2))$ is a solution to \eqref{alternative} with initial value $u(0)=(-\mathcal R_2 \theta_0,\mathcal R_1 \theta_0)$, then $\theta=\mathcal R_2 u_1 - \mathcal R_1 u_2$ is a solution to \eqref{sqg} on $[0,T]$ with $\theta(0)=\theta_0$.
\end{Prop}

\begin{proof}
The first part was already shown in the derivation of \eqref{alternative}. To show the second part we take a solution $u$ to \eqref{alternative} with $u(0)=(-\mathcal R_2 \theta_0,\mathcal R_1 \theta_0)$ on $[0,T]$. As $\operatorname{div} u(0)=0$ we have by Lemma \ref{lemma_div} for all $t \in [0,T]$
\begin{equation}\label{div0}
\mathcal R_1 u_1 + \mathcal R_2 u_2 = 0
\end{equation}
We have to show that $\theta:=\mathcal R_2 u_1 - \mathcal R_1 u_2$ is a solution to \eqref{sqg}. By \eqref{div0} we have
\[
 -\mathcal R_2 \theta = -\mathcal R_2^2 u_1 + \mathcal R_1 \mathcal R_2 u_2 = (-\mathcal R_2^2-\mathcal R_1^2) u_1 = u_1
\]
Similarly we have $\mathcal R_1 \theta = u_2$. Applying $-\mathcal R_2$ to the first equation in \eqref{alternative}, $\mathcal R_1$ to the second equation and sum up we get
\[
\partial_t \theta - \mathcal R_2((u \cdot \nabla) u_1) + \mathcal R_1((u \cdot \nabla)u_2) = -\mathcal R_2([u \cdot \nabla,-\mathcal R_2]\theta) + \mathcal R_1([u \cdot \nabla,\mathcal R_1]\theta)
\]
Simplifying we arrive at
\[
 \partial_t \theta + (u \cdot \nabla) \theta = 0
\]
\end{proof}

\section{Local well-posedness}\label{section_lwp}

The goal of this section is to establish the equivalence of \eqref{sqg} and \eqref{lagrangian}. By Proposition \ref{prop_equivalence} we just have to prove the equivalence between \eqref{alternative} and \eqref{lagrangian}.

\begin{Lemma}\label{lemma_lagrangian_to_eulerian}
Let $s > 2$ and $T > 0$. Assume that $\varphi$ is a solution of \eqref{lagrangian} on $[0,T]$ for the initial values $\varphi(0)=\mbox{id}$ and $v(0)=u_0 \in H^s(\R^2;\R^2)$. Then $u$ given by
\[
 u(t):=\varphi_t(t) \circ \varphi(t)^{-1}
\] 
is a solution to \eqref{alternative}.
\end{Lemma}

\begin{proof}
It is to prove that we have
\[
 u(t) = u_0 + \int_0^t B(u(s),u(s)) - (u(s) \cdot \nabla) u(s) \;ds \quad \forall \; 0 \leq t \leq T
\]
Note that by Proposition \ref{prop_analytic} we have $\varphi \in C^\infty([0,T];\D^s(\R^2))$. Therefore by the properties of the composition (see \cite{composition}) $u \in C([0,T];H^s(\R^2;\R^2))$. By the Sobolev imbedding $H^s \hookrightarrow C^1$ we also have $u=\varphi_t \circ \varphi^{-1} \in C^1([0,T] \times \R^2;\R^2)$. Thus we have pointwise
\[
 \varphi_{tt} = (u_t + (u \cdot \nabla) u ) \circ \varphi
\]
As $\varphi$ is a solution to \eqref{lagrangian} we conclude pointwise
\[
 (u_t + (u \cdot \nabla) u) \circ \varphi = B(u,u) \circ \varphi
\]
or $u_t + (u \cdot \nabla) u = B(u,u)$. Rewriting this we get
\[
 u(t) = u_0 + \int_0^t B(u(s),u(s))-(u(s) \cdot \nabla) u(s) \;ds \quad \forall \; 0 \leq t \leq T
\]
By the Banach algebra property of $H^{s-1}$ and the imbedding $H^{s-1} \hookrightarrow C^0$ the integral is also an identity for $H^{s-1}$ functions.
\end{proof}

The reverse is

\begin{Lemma}\label{lemma_eulerian_to_lagrangian}
Let $s > 2$ and $T > 0$. If $u \in C([0,T];H^s(\R^2;\R^2))$ is a solution to \eqref{alternative} then its flow map $\varphi$ is a solution to \eqref{lagrangian}.
\end{Lemma}

\begin{proof}
We know (see \cite{thesis}) that for $u$ there is a unique $\varphi \in C^1([0,T];\D^s(\R^2))$ with
\[
 \varphi_t = u \circ \varphi,\quad \varphi(0)=\mbox{id}
\]
By the integral relation $u(t)=u_0 + \int_0^t B(u,u)-(u \cdot \nabla) u \;ds$ we see that $u \in C^1([0,T] \times \R^2;\R^2)$. Taking the derivative we get pointwise
\[
  \varphi_{tt} = (u_t + (u \cdot \nabla)u) \circ \varphi = B(u,u) \circ \varphi
\]
This means
\[
 \varphi(t) = \varphi_t(0) + \int_0^t B(\varphi_t \circ \varphi^{-1},\varphi_t \circ \varphi^{-1}) \circ \varphi \; ds \quad \forall \; 0 \leq t \leq T
\]
which is also an identity in $H^s$. Therefore $\varphi \in C^2([0,T];\D^s(\R^2))$ and $\varphi$ is a solution to \eqref{lagrangian}.
\end{proof}

Lemma \ref{lemma_lagrangian_to_eulerian} and Lemma \ref{lemma_eulerian_to_lagrangian} establish the equivalence of \eqref{alternative} and \eqref{lagrangian}. The solutions of \eqref{lagrangian} can be described by an exponential map as follows: Consider \eqref{lagrangian} with initial values $\varphi(0)=\mbox{id}$ and $v(0) \in H^s(\R^2;\R^2)$. Further we denote by $U \subseteq H^s(\R^2;\R^2)$ those initial values $v(0)=u_0$ for which we have a solution on $[0,1]$. With this we define
\[
 \exp: U \to \D^s(\R^2),\quad u_0 \mapsto \varphi(1;u_0)
\]
where $\varphi(1;u_0)$ is the time one value of the solution $\varphi$ corresponding to the initial values $\varphi(0)=\mbox{id}$ and $v(0)=u_0$. By Proposition \eqref{prop_analytic} we know that $\exp$ is real analytic because we have analytic dependence on the initial value $u_0$. The $\varphi$-solution can be totally described by the exponential map. For $u_0$ the corresponding $\varphi$ is just given by 
\[
 \varphi(t)=\exp(tu_0)
\]
for all $t$ for which $tu_0$ lies in $U$. Furthermore the derivative of $\exp$ at $0$ is the identity map. For details on the exponential map one can consult \cite{lang}. We end this section by giving a proof of Theorem \ref{lwp}.

\begin{proof}[Proof of Theorem \ref{lwp}]
Take $\theta_0 \in H^s(\R^2)$ and define for $u_0=(-\mathcal R_2\theta_0,\mathcal R_1 \theta_0)$
\[
 \varphi(t)=\exp(t u_0) \quad \mbox{and} \quad u(t)=\varphi_t(t) \circ \varphi(t)^{-1}
\]
for all $t \geq 0$ with $t u_0$ in the domain of definition of $\exp$. By the properties of the composition map -- see \cite{composition} -- we know that $u \in C([0,T];H^s(\R^2;\R^2))$ for some $T > 0$. With this we define
\[
 \theta(t) = \mathcal R_2 u_1(t) - \mathcal R_1 u_2(t)
\]
which solves \eqref{sqg}. Furthermore the dependence on $\theta_0$ is continuous. Uniqueness of solutions follows from the uniqueness of solutions to ODEs. More precisely assume two solutions $\theta$ and $\tilde \theta$ with the same initial value $\theta_0$. Define the corresponding $u=(-\mathcal R_2\theta,\mathcal R_1 \theta)$ und $\tilde u=(-\mathcal R_2 \tilde \theta,\mathcal R_1 \tilde \theta)$. By Lemma \ref{lemma_eulerian_to_lagrangian} their flows $\varphi$ resp. $\tilde \varphi$ are solutions to \eqref{lagrangian}. So they have to be equal which implies that $\theta \equiv \tilde \theta$.
\end{proof}

\section{Non-uniform dependence}\label{section_nonuniform}

Throughout this section we assume $s > 2$. In this section we prove Theorem \ref{nonuniform}. We introduce the notation $\Phi_T$ for the time $T$-solution map, $T > 0$, i.e. $\Phi_T(\theta_0)$ denotes the value of the solution to \eqref{sqg} with initial value $\theta_0$ at time $T$. As already introduced we denote by $U_T \subseteq H^s(\R^2)$ the domain of definition of $\Phi_T$. In the case of $T=1$ we use $\Phi:=\Phi_T$ and $U:=U_T$. By the scaling property of \eqref{sqg} we have
\[
 \Phi_T(\theta_0)= \frac{1}{T} \Phi(T \theta_0) \quad \mbox{and} \quad U_T=\frac{1}{T} U
\]
So to prove Theorem \ref{nonuniform} it suffices to give a proof for the special case $T=1$.

\begin{Prop}\label{prop_nonuniform}
The map
\[
 \Phi:U \to H^s(\R^2),\quad \theta_0 \mapsto \Phi(\theta_0)
\]
is nowhere locally uniformly continuous.
\end{Prop}

Before proving Proposition \ref{prop_nonuniform} we have to make some preparation.

\begin{Lemma}\label{lemma_nonzero}
For any $x \in \R^2$ there is a $\theta \in H^s(\R^2)$ such that $u(x) \neq 0$ where
\[
 u=(-\mathcal R_2 \theta,\mathcal R_1 \theta)
\]
\end{Lemma}

\begin{proof}
Recall the integral representation for the Riesz transforms $\mathcal R_k, k=1,2$
\[
 \mathcal R_k \theta(x) = \frac{1}{\sqrt{2\pi}} \;p.v. \int_{\R^2} \frac{x_k-y_k}{|x-y|^3} \theta(y) \;dy,\quad k=1,2
\]
in the principal value sense. For the given $x \in \R^2$ we can just choose a smooth positive $\theta$ with compact support lying on the left-down of $x$. We therefore have trivially $\mathcal R_k\theta(x) > 0$ for $k=1,2$.
\end{proof}

A consequence of this lemma is the following technical lemma.

\begin{Lemma}\label{lemma_dense}
There is a dense subset $S \subseteq U$ consisting of functions with compact support such that each function $\theta_0 \in S$ has the following property:\\
There is $x \in \R^2$ and $\theta \in H^s(\R^2)$ depending on $\theta_0$ such that
\[
 \operatorname{dist}(x,\operatorname{supp}\theta_0) \geq 2
\]
and
\[
 (d_{u_0} \exp(u))(x) \neq 0
\]
where $u_0=(-\mathcal R_2 \theta_0,\mathcal R_1 \theta_0), u=(-\mathcal R_2 \theta,\mathcal R_1 \theta)$ and $d_{u_0}\exp$ denotes the differential of $\exp$ at the point $u_0$.
\end{Lemma}

\begin{proof}
Take $\theta_0 \in U$ with compact support. Take an arbitrary $x \in \R^2$ which has a distance of more than $2$ to the support of $\theta_0$. By Lemma \ref{lemma_nonzero} there is $\theta \in H^s(\R^2)$ with $u(x) \neq 0$. Consider now the analytic function
\[
 t \mapsto (d_{t u_0} \exp(u))(x)
\]
At $t=0$ this equals to $u(x) \neq 0$. Therefore there is a sequence $0 \leq t_n \uparrow 1$ with
\[
 (d_{t_n u_0} \exp(u))(x) \neq 0,\quad \forall n \geq 1
\]
We put all $t_n \theta_0$ into $S$. Doing that for all $\theta_0$ with compact support we get our desired result, as the compactly supported functions are dense in $H^s(\R^2)$.
\end{proof}

In the following we will use inequalities for functions with disjoint compact support of the type
\[
 ||f + g|| \geq C(||f|| + ||g||)
\]
for Sobolev norms. More precisely, given $s' \geq 0$ and fixed disjoint compact sets $K_1, K_2 \subseteq \R^2$ there is a constant $C > 0$ such that we have
\begin{equation}\label{ineq1}
 ||f+g||_{s'} \geq C(||f||_{s'} + ||g||_{s'})
\end{equation}
for all $f,g \in H^{s'}(\R^2)$ with $f,g$ supported in $K_1$ resp. $K_2$. We have a similar situation if the geometry of the supports is in a fixed ratio. We will use it as follows: There is a constant $C > 0$ such that for $x,y \in \R^2$ with
\[
 0 < r:=\frac{|x-y|}{4} < 1
\]
we have
\begin{equation}\label{ineq2}
 ||f+g||_{s'} \geq C(||f||_{s'}+||g||_{s'})
\end{equation}
for all $f,g \in H^{s'}(\R^2)$ with $f$ supported in $B_r(x)$ and $g$ supported in $B_r(y)$. Here $B_r(z)$ denotes the ball around $z$ with radius $r$. For the details one can look at the Appendix in \cite{b_family}.\\ \\
Now we prove Proposition \ref{prop_nonuniform}.

\begin{proof}[Proof of Proposition \ref{prop_nonuniform}]
For a given $\theta_0 \in U$ we will show that there is $R_\ast > 0$ such that $\Phi$ is not uniformly continuous on $B_R(\theta_0)$ for all $0 < R \leq R_\ast$. We will choose $R_\ast$ in several steps. It is enough to show that for $\theta_0$ in the dense subset $S \subseteq U$. So take an arbitrary $\theta_0$ in $S$. To make the notation easier we introduce the analytic map $\widetilde \exp$
\[
 \widetilde \exp:U \to \D^s(\R^2),\quad \theta \mapsto \exp((-\mathcal R_2 \theta,\mathcal R_1 \theta))
\]
In particular we then have for a solution $\theta$ of \eqref{sqg}
\begin{equation}\label{conserved_exp}
\theta(1)=\theta(0) \circ (\widetilde\exp(\theta(0)))^{-1}
\end{equation}
Furthermore we fix by Lemma \ref{lemma_dense} a $v \in H^s(\R^2;\R^2), v \neq 0$, and $x^\ast \in \R^2$ with $\operatorname{dist}(x^\ast,\operatorname{supp}\theta_0) \geq 2$ and
\[
 |(d_{\theta_0} \widetilde\exp(v))(x^\ast)| \geq m ||v||_s
\]
for some $m > 0$. By the Sobolev imbedding we have
\begin{equation}\label{sobolev_imbedding}
||f||_{C^1} \leq C_1 ||f||_s,\quad \forall f \in H^s(\R^2)
\end{equation}
for some $C_1 > 0$. We choose an $R_1 > 0$ such that for some $C_2 > 0$ we have
\begin{equation}\label{bdd}
\frac{1}{C_2} ||f||_s \leq ||f \circ \varphi^{-1}||_s \leq C_2 ||f||_s
\end{equation}
for all $f \in H^s(\R^2)$ and $\varphi \in \widetilde\exp(B_{R_1}(\theta_0))$. That this is indeed possible follows from the continuity of the composition map and the linearity in $f$ -- see \cite{composition}. We try to get a situation as described in \eqref{ineq1}. Let $\varphi_0=\widetilde \exp(\theta_0)$. Then as $x^\ast$ is enough away from the support of $\theta_0$ we have
\[
 d:=\operatorname{dist}\left(\varphi_0(\operatorname{supp}\theta_0),\varphi_0(\overline{B_1(x^\ast)})\right) > 0
\]
We choose $K_1, K_2$ as
\[
 K_1:=\{ x \in \R^2 \; | \; \operatorname{dist}(x,\varphi_0(\operatorname{supp}\theta_0)) \leq d/4 \}
\]
resp.
\[
 K_2:=\{ x \in \R^2 \; | \; \operatorname{dist}(x,\varphi_0(\overline{B_1(x^\ast)})) \leq d/4 \}
\]
By choosing $0 < R_2 \leq R_1$ we can ensure by the Sobolev imbedding \eqref{sobolev_imbedding}
\[
 |\varphi(x)-\varphi(y)| \leq L |x-y|,\quad \forall x,y \in \R^2 \quad \mbox{and} \quad ||\varphi-\tilde \varphi||_{L^\infty} \leq \min\{d/4,1\}
\]
for all $\varphi, \tilde \varphi \in \widetilde\exp(B_{R_2}(\theta_0))$. So the second condition ensures
\[
 \varphi(\operatorname{supp}\theta_0) \subseteq K_1 \quad \mbox{and} \quad \varphi(\overline{B_1(x^\ast)}) \subseteq K_2
\]
for all $\varphi \in \widetilde\exp(B_{R_2}(\theta_0))$. Consider the Taylor expansion
\[
 \widetilde\exp(\theta + h) = \widetilde\exp(\theta) + d_\theta \widetilde\exp(h) + \int_0^1 (1-t) d_{\theta + t h}^2 \widetilde\exp(h,h) \;dt
\]
We need to estimate the terms appearing in this expansion. We can choose $0 < R_3 \leq R_2$ in such a way that we have for some constant $K > 0$
\[
 ||d_\theta^2 \widetilde\exp(h_1,h_2)||_s \leq K ||h_1||_s ||h_2||_s
\]
and
\[
 ||d_{\theta_1}^2 \widetilde \exp(h_1,h_2)-d_{\theta_2}^2 \widetilde\exp(h_1,h_2)||_s \leq K ||\theta_1-\theta_2||_s ||h_1||_s ||h_2||_s
\]
for all $\theta, \theta_1, \theta_2 \in B_{R_3}(\theta_0)$ and $h_1,h_2 \in H^s(\R^2)$. Finally we choose $0 < R_\ast \leq R_3$ with
\begin{equation}\label{condition}
\max\{C_1 K R_\ast,C_1 K R_\ast^2\} < m/8
\end{equation}
Now take an arbitrary $0 < R \leq R_\ast$. We will construct two sequences of initial values
\[
 (\theta_0^{(n)})_{n \geq 1}, (\tilde \theta_0^{(n)})_{n \geq 1} \subseteq B_R(\theta_0)
\]
with $\lim_{n \to \infty} ||\theta_0^{(n)}-\tilde \theta_0^{(n)}||_s=0$ but
\[
 \limsup_{n \geq 1} ||\Phi(\theta_0^{(n)})-\Phi(\tilde \theta_0^{(n)})||_s > 0
\]
showing the claim. The first sequence will be chosen in the form
\[
 \theta_0^{(n)}=\theta_0 + w^{(n)}
\]
where we take $w^{(n)} \in H^s(\R^2)$ arbitrarily with $||w^{(n)}||_s=R/2$ and having its support in $B_{r_n}(x^\ast)$ where
\[
 r_n=\frac{m}{8nL}||v||_s
\]
Thus the mass of $w^{(n)}$ is constant whereas its support shrinks to $x^\ast$. The second sequence is a perturbation of the first one so as to get a shift in the supports. We take it as
\[
 \tilde \theta_0^{(n)} = \theta_0^{(n)} + \frac{1}{n}v=\theta_0 + w^{(n)}+\frac{1}{n} v
\]
We will use the notation $v^{(n)}:=\frac{1}{n}v$. Taking $N$ large enough we clearly have
\[
 \theta_0^{(n)},\tilde \theta_0^{(n)} \in B_R(\theta_0) \mbox{ and } r_n \leq 1, \quad \forall n \geq N
\]
By construction we have
\[
 \lim_{n \to \infty} ||\theta_0^{(n)}-\tilde \theta_0^{(n)}||_s = \lim_{n \to \infty} ||v^{(n)}||_s = 0
\]
We introduce for $n \geq N$
\[
 \varphi^{(n)} = \widetilde \exp(\theta_0^{(n)}) \quad \mbox{resp.} \quad \tilde \varphi^{(n)} = \widetilde \exp(\tilde \theta_0^{(n)})
\]
By the conservation law \eqref{conserved_exp} we have
\[
 \Phi(\theta_0^{(n)})=\theta_0^{(n)} \circ (\varphi^{(n)})^{-1} \quad \mbox{resp.} \quad \Phi(\tilde \theta_0^{(n)})=\tilde \theta_0^{(n)} \circ (\tilde \varphi^{(n)})^{-1}
\]
We will use these expressions to evaluate $||\Phi(\theta_0^{(n)})-\Phi(\tilde \theta_0^{(n)})||_s$. Plugging in the expressions we get
\begin{eqnarray*}
&& ||\Phi(\theta_0^{(n)})-\Phi(\tilde \theta_0^{(n)})||_s = ||(\theta_0+w^{(n)}) \circ (\varphi^{(n)})^{-1} - (\theta_0+w^{(n)}+v^{(n)}) \circ (\tilde \varphi^{(n)})^{-1}||_s \\
&& \geq ||(\theta_0+w^{(n)}) \circ (\varphi^{(n)})^{-1}-(\theta_0 + w^{(n)}) \circ (\tilde \varphi^{(n)})^{-1} ||_s - ||v^{(n)} \circ (\tilde \varphi^{(n)})^{-1}||_s
\end{eqnarray*}
By \eqref{bdd} the last expression vanishes if we take the $\limsup$. So we just have to look at the first term on the right
\begin{eqnarray*}
&&||(\theta_0+w^{(n)}) \circ (\varphi^{(n)})^{-1}-(\theta_0 + w^{(n)}) \circ (\tilde \varphi^{(n)})^{-1} ||_s =\\
&& ||(\theta_0 \circ (\varphi^{(n)})^{-1}-\tilde \theta_0 \circ (\tilde \varphi^{(n)})^{-1})+(w^{(n)} \circ (\varphi^{(n)})^{-1}-w^{(n)} \circ (\tilde \varphi^{(n)})^{-1})||_s
\end{eqnarray*}
The first two terms in the latter expression have their support in $K_1$ and the other two in $K_2$. By \eqref{ineq1} it will be enough to establish
\[
 \limsup_{n \to \infty} ||w^{(n)} \circ (\varphi^{(n)})^{-1} - w^{(n)} \circ (\tilde \varphi^{(n)})^{-1}||_s > 0
\]
We will do that by showing that the supports of these two expressions are disjoint in a way that we can apply \eqref{ineq2}. To do that we will estimate the distance $|\varphi^{(n)}(x^\ast)-\tilde \varphi^{(n)}(x^\ast)|$ using the Taylor expansion of $\widetilde \exp$. So we have
\begin{eqnarray*}
\varphi^{(n)} &=& \widetilde \exp(\theta_0 + w^{(n)}) \\
&=& \widetilde \exp(\theta_0) + d_{\theta_0} \widetilde \exp(w^{(n)}) + \int_0^1 (1-t) d_{\theta_0 + t w^{(n)}}^2 \widetilde \exp(w^{(n)},w^{(n)}) \;dt
\end{eqnarray*}
Similarly 
\begin{eqnarray*}
&& \tilde \varphi^{(n)} = \widetilde \exp(\theta_0 + w^{(n)} + v^{(n)}) = \widetilde \exp(\theta_0) + d_{\theta_0} \widetilde \exp(w^{(n)}+v^{(n)}) \\
&&+ \int_0^1 (1-t) d_{\theta_0+t(w^{(n)}+v^{(n)})}^2 \widetilde \exp(w^{(n)}+v^{(n)},w^{(n)}+v^{(n)})\;dt
\end{eqnarray*}
So the difference reads as
\[
 \varphi^{(n)}-\tilde \varphi^{(n)} = -d_{\theta_0} \widetilde \exp(v^{(n)}) + I_1 + I_2 + I_3
\]
where
\[
 I_1 = \int_0^1 (1-t) \left(d_{\theta_0 + tw^{(n)}}^2\widetilde \exp(w^{(n)},w^{(n)}) - d_{\theta_0 + t(w^{(n)}+v^{(n)})}^2\widetilde \exp(w^{(n)},w^{(n)})\right) \; dt
\]
and
\[
 I_2= -2 \int_0^1 (1-t) d_{\theta_0+t(w^{(n)}+v^{(n)})}^2\widetilde \exp(v^{(n)},w^{(n)}) \;dt
\]
and
\[
 I_3= - \int_0^1 (1-t) d_{\theta_0 + t(w^{(n)}+v^{(n)})}^2 \widetilde\exp(v^{(n)},v^{(n)}) \;dt
\]
Using the estimates for the second derivatives from above we get
\[
 ||I_1||_s \leq K ||v^{(n)}||_s ||w^{(n)}||_s^2 = \frac{K}{4n} ||v||_s R^2
\]
and
\[
 ||I_2||_s \leq 2K ||v^{(n)}||_s ||w^{(n)}||_s = \frac{K}{n} ||v||_s R
\]
and
\[
 ||I_3||_s \leq \frac{K}{n} ||v||_s \frac{||v||_s}{n} \leq \frac{KR}{n} ||v||_s
\]
where the last inequality holds for $n \geq N$ by enlarging $N$ if necessary. Thus we see by the Sobolev imbedding that the value at $x^\ast$ can be estimated by
\[
 |I_1(x^\ast)|+|I_2(x^\ast)|+|I_3(x^\ast)| \leq \frac{C_1 K R^2}{4n} ||v||_s + \frac{C_1 K R}{n} ||v||_s + \frac{C_1 K R}{n} ||v||_s
\]
By the choice for $R_\ast$ it follows from  \eqref{condition}
\[
 |I_1(x^\ast)|+|I_2(x^\ast)|+|I_3(x^\ast)| \leq \frac{m}{2n} ||v||_s
\]
Using this inequality we arrive at
\[
 |\varphi^{(n)}(x^\ast)-\tilde \varphi^{(n)}(x^\ast)| \geq |d_{\theta_0}\widetilde \exp(v^{(n)})(x^\ast)| - \frac{m}{2n}||v||_s
\]
Hence
\[
 |\varphi^{(n)}(x^\ast)-\tilde \varphi^{(n)}(x^\ast)| \geq \frac{1}{n} m ||v||_s - \frac{m}{2n}||v|_s = \frac{m}{2n} ||v||_s
\]
By the Lipschitz property for $\varphi^{(n)},\tilde \varphi^{(n)}$ we have
\[
 \varphi^{(n)}(B_{r_n}(x^\ast)) \subseteq B_{R_n}(\varphi^{(n)}(x^\ast))
\]
with $R_n=L r_n=L \frac{m}{8nL}||v||_s=\frac{m}{8n} ||v||_s$. Similarly
\[
 \tilde \varphi^{(n)}(B_{r_n}(x^\ast)) \subseteq B_{R_n}(\tilde \varphi^{(n)}(x^\ast))
\]
This means $w^{(n)} \circ (\varphi^{(n)})^{-1}$ is supported in $B_{R_n}(\varphi^{(n)}(x^\ast))$ and $w^{(n)} \circ (\tilde \varphi^{(n)})^{-1}$ is supported in $B_{R_n}(\tilde \varphi^{(n)}(x^\ast))$. So we are in a situation where we can apply \eqref{ineq2} since the distance between the centers of support is larger that $\frac{m}{2n}||v||_s$ and the radii of the supports are $\frac{m}{8n}||v||_s$. Thus we have
\begin{eqnarray*}
&& ||w^{(n)} \circ (\varphi^{(n)})^{-1} - w^{(n)} \circ (\tilde \varphi^{(n)})^{-1}||_s \\
&& \geq C (||w^{(n)} \circ (\varphi^{(n)})^{-1}||_s + ||w^{(n)} \circ (\tilde \varphi^{(n)})^{-1}||_s) \geq \frac{C}{C_2} R/2
\end{eqnarray*}
where we used \eqref{bdd}. Thus we have
\[
 \limsup_{n \to \infty} ||\Phi(\theta_0^{(n)})-\Phi(\tilde \theta_0^{(n)})||_s \geq \tilde C R
\]
with $\tilde C$ independent of $0 < R \leq R_\ast$ whereas $\lim_{n \to \infty} ||\theta_0^{(n)}-\tilde \theta_0^{(n)}||_s = 0$. As this holds for every $0 < R \leq R_\ast$ we are done.
\end{proof}

\appendix

\section{Proof of Proposition \ref{prop_analytic}}\label{section_appendix}

In this section we will prove Proposition \ref{prop_analytic}. The ideas we will use are inspired by \cite{chemin}, \cite{constantin_sqg} and \cite{lagrangian}. Throughout this section we assume $s > n/2+1$. We introduce the operator
\[
 \Lambda = (-\Delta)^{1/2}
\]
So the Fourier transform of $\Lambda f$ is given by $|\xi| \hat f(\xi)$ where $\hat f$ denotes the Fourier transform of $f$. In the following we will also use the definition of $\Lambda$ in terms of a principal value integral
\[
 \Lambda f(x) = \mbox{p.v. } \int_{\R^n} \frac{f(x)-f(y)}{|x-y|^{n+1}} \;dy = \lim_{\varepsilon \to 0} \int_{|x-y| \geq \varepsilon} \frac{f(x)-f(y)}{|x-y|^{n+1}} \;dy
\]
We will use also the following regularization of the above singular integral -- see \cite{samko} for the technical details -- for $f$ regular enough
\begin{eqnarray*}
 \Lambda f(x) &=&  \int_{\R^n} \frac{f(x)-f(y)+(x-y)\cdot \nabla f(x)}{|x-y|^{n+1}} e^{-|x-y|^2} \;dy + \int_{\R^n} \frac{f(x)-f(y)}{|x-y|^{n+1}} (1-e^{-|x-y|^2})\;dy\\
&=& \int_{\R^n} K_1(x-y) (f(x)-f(y)+(x-y) \cdot \nabla f(x))\;dy + \int_{\R^n} K_2(x-y) (f(x)-f(y)) \;dy
\end{eqnarray*}
From \cite{constantin_sqg} we can deduce that the analytic functions $K_1(y),K_2(y)$ satisfy the estimates
\begin{equation}\label{k1_estimate}
 |\partial^\alpha K_1(y)| \leq \frac{C^{|\alpha|} |\alpha|!}{|y|^{n+1+|\alpha|}} e^{-|y|^2/2}
\end{equation}
resp.
\begin{equation}\label{k2_estimate}
 |\partial^\alpha K_2(y)| \leq C^{|\alpha|} |\alpha|! \min\{\frac{1}{|y|^{n-1+|\alpha|}},\frac{1}{|y|^{n+1+|\alpha|}}\}
\end{equation}
for all $\alpha \in \N^n$ and some universal constant $C > 0$. In the following we will also often use the algebra property of Sobolev spaces. Making the above $C$ larger if necessary we have the Kato-Ponce inequality
\begin{equation}\label{banach_algebra}
 ||f \cdot g||_{s-1} \leq C(||f||_{s-1} ||g||_\infty + ||f||_\infty ||g||_{s-1})
\end{equation}
and also (can be deduced from \eqref{banach_algebra})
\[
 ||f \cdot g||_{s-1} \leq C ||f||_{s-1} ||g||_{s-1}
\]
We prove Proposition \ref{prop_analytic} in several steps.

\begin{Lemma}\label{lemma_laplacian}
The map 
\[
 \Ds^s(\R^n) \to L(H^s(\R^n);H^{s-1}(\R^n)),\quad \varphi \mapsto [f \mapsto R_\varphi \Lambda R_\varphi^{-1} f]
\]
is real analytic.
\end{Lemma} 

For the concept of analyticity in Hilbert spaces one can consult \cite{lagrangian}

\begin{proof}
The goal is to establish a power series expansion of
\[
 \left(\Lambda f(\varphi^{-1}(x))\right) \circ \varphi(x)
\]
in terms of $g=(g_1,\ldots,g_n)$ where $\varphi=\mbox{id}+g$. We split this expression according to the above regularization as
\[
R_\varphi \Lambda R_\varphi^{-1} f = \mathcal K_1(\varphi)+\mathcal K_2(\varphi)
\]
We first treat the easier case $\mathcal K_2(\varphi)$. We have
\[
 \mathcal K_2(\varphi) = \int_{\R^n} K_2(\varphi(x)-\varphi(y)) (f(x)-f(y)) J_\varphi(y) \;dy
\]
where $J_\varphi$ is the determinant of the Jacobian $d\varphi$. Note that $J_\varphi$ is a fixed polynomial in the first derivatives of $g$. We use $\varphi(x)-\varphi(y)=(x-y)+(g(x)-g(y))$ and expand into the Taylor series of $K_2$
\begin{eqnarray*}
 \mathcal K_2(\varphi) &=& \int_{\R^n} \sum_{\alpha \in \N^n} \frac{1}{\alpha!} \partial^\alpha K_2(x-y) (g(x)-g(y))^\alpha (f(x)-f(y)) J_\varphi(y) \;dy\\
&=& \int_{\R^n} \sum_{\alpha \in \N^n} \frac{1}{\alpha!} \partial^\alpha K_2(y) (g(x)-g(x-y))^\alpha (f(x)-f(x-y)) J_\varphi(x-y) \;dy\\
\end{eqnarray*}
If we separate the monomials in $J_\varphi$ we see that the individual terms are multilinear expressions in $g$. Taking the $H^{s-1}$-norm in the individual summand and using the Banach algebra properties we can dominate it by
\[
 \int_{\R^n} \frac{1}{\alpha!} |\partial^\alpha K_2(y)| C^{|\alpha|} ||g(x)-g(x-y)||_{s-1}^{|\alpha|} 2 ||f||_{s-1} K(1+||g||_s^n) \; dy
\]
for some fixed $K > 0$. Using $||g(x)-g(x-y)||_{s-1} \leq ||g||_s |y|$ and \eqref{k2_estimate} we can estimate this by
\[
 \frac{|\alpha|!}{\alpha!}2 C^{2|\alpha|} K ||g||_s^{|\alpha|} ||f||_s (1+||g||_s^n) \int_{\R^n} \min\{\frac{1}{|y|^{n-1}},\frac{1}{|y|^{n+1}}\} \;dy
\]
Summing over all $\alpha$ with $|\alpha|=k$ for a fixed $k \in \N$ we have an upper bound
\[
 \tilde C^k ||g||_s^k ||f||_s (1+||g||_s^n)
\]
which is the general term in a convergent series for small $||g||_s$, i.e. for $\varphi$ near to the identity map $\mbox{id}$. Now consider $\mathcal K_1(\varphi)$. We have
\[
 \mathcal K_1(\varphi) = \int_{\R^n} K_1(\varphi(x)-\varphi(y)) (f(x)-f(y)+(\varphi(x)-\varphi(y))^\top \cdot [d\varphi^{-1}(x)]^\top \cdot \nabla f(x)) J_\varphi(y) \;dy
\]
Developping into the Taylor series of $K_1$ as above we have
\[
 \int_{\R^n} \sum_{\alpha \in \N^n} \frac{1}{\alpha!} \partial^\alpha K_1(x-y) (g(x)-g(y))^\alpha (f(x)-f(y)+(\varphi(x)-\varphi(y))^\top \cdot [d\varphi^{-1}(x)]^\top \cdot \nabla f(x)) J_\varphi(y) \;dy
\]
By pulling out $1/J_\varphi(x)$ in front of the integral one sees by the formula for the inverse of a matrix that the individual terms under the integral are polynomials in $g$. Note that $1/J_\varphi$ depends analytically on $\varphi$ -- see \cite{lagrangian}. A change of variables leads to
\begin{multline*}
  \int_{\R^n} \sum_{\alpha \in \N^n} \frac{1}{\alpha!} \partial^\alpha K_1(y) (g(x)-g(x-y))^\alpha \\ (f(x)-f(x-y)+(\varphi(x)-\varphi(x-y))^\top \cdot [d\varphi^{-1}(x)]^\top \cdot \nabla f(x)) J_\varphi(x-y) \;dy
\end{multline*}
In order to get integrability of the kernels we need to replace $g(x)-g(x-y)$ by terms which are of higher order than 1. Therefore we write
\[
 g(x)-g(x-y)=g(x)-g(x-y) + dg(x) \cdot y - dg(x) \cdot y = R_g(x,y) - dg(x) \cdot y
\]
The $R_g(x,y)$ term is convenient since we have
\[
 g(x)-g(x-y) + dg(x) \cdot y = \int_0^1 (dg(x)-dg(x-ty)) \cdot y \; dt
\]
Thus we have the estimates
\[
 ||R_g(x,y)||_{H^{s-1}(dx)} \leq 2 ||g||_s |y|
\]
and
\[
 ||R_g(x,y)||_{L^\infty(dx)} \leq C ||g||_s |y|^{1+\varepsilon} 
\]
for some $0 < \varepsilon < 1$ because of the Sobolev imbedding $H^{s-1} \hookrightarrow C^\varepsilon$. We have similarly
\begin{multline*}
R_{f,\varphi}(x,y):=(f(x)-f(x-y)+(\varphi(x)-\varphi(x-y))^\top \cdot [d\varphi^{-1}(x)]^\top \cdot \nabla f(x))=\\
(f \circ \varphi^{-1})(\varphi(x))-(f \circ \varphi^{-1})(\varphi(x-y)) + (\varphi(x)-\varphi(x-y))^\top \cdot \nabla (f \circ \varphi^{-1}) (\varphi(x))
\end{multline*}
Now if we restrict $\varphi$ to a small ball we can assume -- see \cite{composition}
\[
 ||f \circ \varphi||_s \leq C ||f||_s \quad \mbox{and} \quad |\varphi(x)-\varphi(x-y)| \leq C |y|
\]
for all $f \in H^s$ and $\varphi$ in this ball. With that we get the same estimates 
\[
 ||R_{f,\varphi}(x,y)||_{H^{s-1}(dx)} \leq C ||f||_s |y|
\]
resp.
\[
 ||R_{f,\varphi}(x,y)||_{L^\infty(dx)} \leq C ||f||_s |y|^{1+\varepsilon}
\]
Using this notation the individual term in the integral looks like
\[
 \int_{\R^n} \frac{1}{\alpha!} \partial^\alpha K_1(y) (R_g(x,y)-dg(x) \cdot y)^\alpha R_{f,\varphi}(x,y) J_\varphi(x-y) \; dy
\]
Expanding the bracket we see that $2^{|\alpha|}-1$ terms appear with at least two $R$ terms and one with one $R$ term. Using \eqref{banach_algebra} and \eqref{k1_estimate} one can estimate the $H^{s-1}$ norm of the $2^{|\alpha|}-1$ integrals as
\[
 \int_{\R^n} \frac{|\alpha|!}{\alpha!} C^{|\alpha|} \frac{1}{|y|^{n+1+|\alpha|}} 2^{|\alpha|} e^{-|y|^2/2} C^{|\alpha|} |y|^{1+|\alpha|+\varepsilon} ||g||_s^{|\alpha|} ||f||_s (1+||g||_s^n) \; dy 
\]
Summing over $\alpha$ with $|\alpha|=k$ we have the bound
\[
 \tilde C^{k} ||g||_s^{k} ||f||_s (1+||g||_s^n) \int_{\R^n} \frac{1}{|y|^{n-\varepsilon}} e^{-|y|^2/2} \;dy
\]
which is the general term for a convergent series for $||g||_s$ small, i.e. for $\varphi$ near $\mbox{id}$. The remaining term we have to consider is
\[
 \int_{\R^n} \frac{1}{\alpha!} \partial^\alpha K_1(y) (-dg(x) \cdot y)^\alpha R_{f,\varphi}(x,y) J_\varphi(x-y) \; dy
\]
Expanding the bracket gives $n^{|\alpha|}$ terms of the form
\[
 (-1)^{|\alpha|} \partial_{k_1} g^{m_1}(x) \cdots \partial_{k_{|\alpha|}} g^{m_{|\alpha|}}(x) \int_{\R^n} \frac{1}{\alpha!} \partial^\alpha K_1(y) y_{k_1} \cdots y_{k_{|\alpha|}}  R_{f,\varphi}(x,y) J_\varphi(x-y) \; dy
\]
for certain $1 \leq k_1,\ldots,k_{|\alpha|},m_1,\ldots,m_{|\alpha|} \leq n$. Let us introduce $\tilde K_j(y) = \partial^\alpha K_j(y) y_{k_1} \cdots y_{k_{|\alpha|}}, j=1,2$ and $\tilde K=\tilde K_1 + \tilde K_2$. So we have to examine after a change of variables
\[
 \int_{\R^n} \tilde K_1(x-y) (f(x)-f(y)+(\varphi(x)-\varphi(y))^\top \cdot [d\varphi^{-1}(x)]^\top \cdot \nabla f(x)) J_\varphi(y) \;dy
\]
We claim that 
\[
  \mbox{p.v. } \int_{\R^n} \tilde K_1(x-y) (f(x)-f(y)) J_\varphi(y) \;dy
\]
resp.
\[
 \mbox{p.v. } \int_{\R^n} \tilde K_1(x-y) ((\varphi(x)-\varphi(y))^\top \cdot [d\varphi^{-1}(x)]^\top \cdot \nabla f(x)) J_\varphi(y) \;dy
\]
exist seperately. Let's consider the first one. By \eqref{k2_estimate} we see from the considerations regarding $\mathcal K_2(\varphi)$ that
\[
 ||\int_{\R^n} \tilde K_2(x-y) (f(x)-f(y)) J_\varphi(y) \;dy||_{H^{s-1}(dx)} \leq |\alpha|! C^{|\alpha|} ||f||_s (1+||g||_s^n) 
\]
holds. Therefore its enough to consider
\[
  \mbox{p.v. } \int_{\R^n} \tilde K(x-y) (f(x)-f(y)) J_\varphi(y) \;dy
\]
But as $J_\varphi(y)=1+j(y)$ with some $j \in H^{s-1}$ (note that $\varphi(y)=y+g(y)$) we can apply Lemma \ref{first_order_kernel} for $\tilde K$ to get
\[
 ||\int_{\R^n} \tilde K(x-y) (f(x)-f(y)) J_\varphi(y) \;dy||_{H^{s-1}(dx)} \leq |\alpha|! C^{|\alpha|} ||f||_s (1+||g||_s^n) 
\]
We split the second principal value integral into
\[
 \left(\mbox{p.v. } \int_{\R^n} \tilde K_1(x-y) (x-y)^\top  J_\varphi(y) \;dy \right) \cdot [d\varphi^{-1}(x)]^\top \cdot \nabla f(x)
\] 
and
\[
 \left(\mbox{p.v. } \int_{\R^n} \tilde K_1(x-y) (g(x)-g(y))^\top  J_\varphi(y) \;dy \right) \cdot [d\varphi^{-1}(x)]^\top \cdot \nabla f(x)
\]
One can handle the second part exactly as above, now $f$ replaced by $g$. To handle the first part note that by the Leibniz rule the only term  which is critical is the one where all the derivatives fall on the $1/|y|^{n+1}$ term in the expression
\[
\tilde k(y):= \partial^\alpha  (\frac{1}{|y|^{n+1}}) e^{-|y|^2} y_{k_1} \cdots y_{k_{|\alpha|}} y
\]
The others are integrable and can be estimated alltogether by
\[
 C^{|\alpha|} |\alpha|! (1+||g||_s^n) \int_{\R^n} \frac{1}{|y|^{n-1}} e^{-|y|^2/2}\; dy
\]
So the only remaining integral is
\[
 \mbox{p.v. } \int_{\R^n} (\partial^\alpha_x  \frac{1}{|x-y|^{n+1}} e^{-|x-y|^2})(x-y)_{k_1} \cdots (x-y)_{k_{|\alpha|}} (x-y) (1+j(y)) \; dy
\]
But from \cite{chemin} we know that for any sphere $S$ around $0$ we have
\[
 \int_{S} \partial^\alpha  \frac{1}{|y|^{n+1}} y_{k_1} \cdots y_{k_{|\alpha|}} y
\; dS(y) = 0
\] 
Therefore the $1$ term in $1+j(y)$ vanishes. Finally consider
\[
 \mbox{p.v. } \int_{\R^n} (\partial^\alpha  \frac{1}{|x-y|^{n+1}}) e^{-|x-y|^2}(x-y)_{k_1} \cdots (x-y)_{k_{|\alpha|}} (x-y) j(y) \; dy
\]
The Fourier transform of the kernel above is 
\[
 \mathcal F[\tilde k(y)] \ast F[e^{-|y|^2}] (\xi) = \int_{\R^n} \mathcal F[\tilde k](\xi-\eta) e^{-|\eta|^2} \;d\eta 
\]
which by the calculations of \ref{first_order_kernel} is seen to be bounded by
\[
 |\mathcal F[\tilde k(y)] \ast F[e^{-|y|^4}] (\xi)| \leq C^{|\alpha|} |\alpha|!
\]
Thus the principal value integral is just a Fourier multiplicator operator with a bounded multiplier acting on $j(y)$. Thus
\begin{multline*}
 || \mbox{p.v. } \int_{\R^n} (\partial^\alpha  \frac{1}{|x-y|^{n+1}}) e^{-|x-y|^2}(x-y)_{k_1} \cdots (x-y)_{k_{|\alpha|}} (x-y) j(y) \; dy||_{H^{s-1}(dx)} \\
\leq C^{|\alpha|} |\alpha|! ||j||_{s-1} \leq C^{|\alpha|} |\alpha|! ||g||_s^n
\end{multline*}
So far we have proved that
\[
 \Ds^s(\R^n) \to L(H^s(\R^n);H^{s-1}(\R^n)),\quad \varphi \mapsto R_\varphi \Lambda R_\varphi^{-1}
\]
is analytic around the identity map $\mbox{id}$, i.e. we have a power series expansion in $g=\varphi - \mbox{id}$
\[
 R_\varphi \Lambda R_\varphi^{-1} = \sum_{k \geq 0} P_k(g,\ldots,g)
\]
where $P_k$ is a continuous homogeneous polynomial of degree $k$ with values in $L(H^s(\R^n);H^{s-1}(\R^n))$. The series has a radius of convergence $R > 0$ which means
\[
 \sup_{k \geq 0} ||P_k|| r^k < \infty,\quad \forall 0 \leq r < R
\]
where $||P_k||$ is the norm given by
\[
 \sup_{||g||_s \leq 1,||f||_s \leq 1} ||P_k(g,\ldots,g)(f)||_{s-1}
\]
We have to prove that $R_\varphi \Lambda R_\varphi^{-1}$ is analytic around any $\varphi_\bullet \in \Ds^s(\R^n)$. We do the calculations first in the smooth category (e.g. $H^\infty$). Taking the derivative at $\varphi_\bullet$ in direction of $g$ we get
\[ 
 R_{\varphi_\bullet} [g \circ \varphi_\bullet^{-1},\Lambda] \nabla (f \circ \varphi_\bullet^{-1})=R_{\varphi_\bullet} P_1(g \circ \varphi_\bullet^{-1})(f \circ \varphi_\bullet^{-1})
\]
Similarly the higher derivatives look like
\[
 k! R_{\varphi_\bullet} P_k(g \circ \varphi_\bullet^{-1},\ldots,g \circ \varphi_\bullet^{-1})(f \circ \varphi_\bullet^{-1})
\]
These are polynomials which can be extended continuously to all $g,f \in H^s$. Now we have in the smooth category for $\varphi=\varphi_\bullet + g$ the identity\[
 R_\varphi \Lambda R_\varphi^{-1}f - R_{\varphi_\bullet} \Lambda R_{\varphi_\bullet}^{-1}f = \int_0^t R_{\varphi_\tau} P_1(g \circ \varphi_\tau^{-1})(f \circ \varphi_\tau^{-1})\; d\tau
\]
where $\varphi_\tau=\varphi_\bullet + \tau g, 0 \leq \tau \leq 1$. By continuity we can extend this to all $g,f \in H^s$ and $\varphi,\varphi_\bullet \in \Ds^s$. So one can conclude that $\varphi \mapsto R_\varphi \Lambda R_\varphi^{-1}f$ is $C^1$ with derivative $R_\varphi P_1(g \circ \varphi^{-1})(f \circ \varphi^{-1})$. Inductively one then shows that it is actually $C^\infty$ with the corresponding derivatives. That $\varphi \mapsto R_\varphi \Lambda R_\varphi^{-1}$ is smooth follows now from general principles -- see \cite{lang}. Choosing $C > 0$ with $||h \circ \varphi_\bullet||_s \leq C ||h||_s, ||h \circ \varphi_\bullet^{-1}||_s \leq C ||h||_s$ for all $h \in H^s$ gives
\[
 ||R_{\varphi_\bullet} P_k(g \circ \varphi_\bullet^{-1},\ldots,g \circ \varphi_\bullet^{-1})(f \circ \varphi_\bullet^{-1})||_{s-1} \leq C^{k+2} ||P_k|| ||g||_s^k ||f||_s 
\] 
This shows the convergence of the Taylor series. Thus $\varphi \mapsto R_\varphi \Lambda R_\varphi^{-1}$ is analytic.
\end{proof}

\begin{Lemma}\label{first_order_kernel}
Let $K$ be the function
\[
 K:y=(y_1,\ldots,y_n) \mapsto (\partial^\alpha_y \frac{1}{|y|^{n+1}}) y^\beta
\]
with $\alpha, \beta \in \N^n$, $|\alpha|=|\beta|=k$. There is $C>0$ independent of $k$ with \[
 ||\mbox{p.v. } \int_{\R^n} K(x-y) (f(x)-f(y)) \;dy ||_{H^{s-1}(dx)} \leq k! C^k ||f||_s
\]
and
\[
 ||\mbox{p.v. } \int_{\R^n} K(x-y) (f(x)-f(y)) g(y) \;dy ||_{H^{s-1}(dx)} \leq k! C^k ||f||_s ||g||_{s-1}
\]
for all $f \in H^s(\R^n)$, $g \in H^{s-1}(\R^n)$.
\end{Lemma}

\begin{proof}
The Fourier Transform of $K$ is given by
\[
 \hat K(\xi)=\frac{1}{(-i)^{|\beta|}} \partial_\xi^\beta \widehat{\partial_y^\alpha \frac{1}{|y|^{n+1}}} (\xi)= \frac{i^{|\alpha|}}{(-i)^{|\beta|}} \partial_\xi^\beta (\xi^\alpha |\xi|)= (-1)^k \partial_\xi^\beta (\xi^\alpha |\xi|)
\]
Thus one has with the same derivation as for \eqref{k1_estimate} and \eqref{k2_estimate}
\begin{equation}\label{K_estimate}
 |\hat K(\xi)| \leq k! C^k |\xi|
\end{equation}
Adjusting $C$ one has in a similar fashion
\[
 |\nabla_\xi \hat K(\xi)| \leq k! C^k
\]
This implies in particular $|K(\xi)-K(\eta)| \leq k! C^k |\xi-\eta|$ which will be used later. Now note that the first principal value integral is nothing other than the Fourier multiplier $K(D)$ acting on $f$. Thus using \eqref{K_estimate} we can estimate this integral by
\[
 ||K(D)f||_{s-1} = ||(1+|\xi|^2)^{(s-1)/2} \hat K(\xi) \hat f(\xi)||_{L^2} \leq k! C^k ||f||_s
\]
The second integral is equal to the commutator $[f,K(D)]g$ which can be seen using the integral representation of $K(D)$ above. Taking the Fourier transform we get
\[
 \mathcal F[ [f,K(D)]g ](\xi) = \mathcal F[f \cdot K(D)g](\xi) - K(\xi) \mathcal F[fg](\xi) = \int_{\R^n} f(\xi - \eta) (K(\eta)-K(\xi)) \hat g(\eta) \;d\eta    
\]
Thus using $|K(\eta)-K(\xi)|\leq k!C^k |\xi-\eta|$ we can bound
\[
 |\mathcal F[ [f,K(D)]g ](\xi)| \leq k!C^k |\hat{f'}| \ast |\hat{g}| (\xi)
\]
where $f'$ is defined by $\hat f'(\xi)=|\xi| \hat f(\xi)$. Taking the inverse Fourier transform we see that this convolution is a multiplication of $H^{s-1}$ functions. Therefore we can bound the second principal value integral by
\[
 ||[f,K(D)]g||_{s-1} \leq C k! C^k ||f||_s ||g||_{s-1}
\] 
which is the desired result after adjusting $C$.
\end{proof}

\begin{Coro}\label{coro_riesz}
For $1 \leq k \leq n$ the map 
\[
 \Ds^s(\R^n) \to L(H^s(\R^n);H^s(\R^n)),\quad \varphi \mapsto [f \mapsto R_\varphi \mathcal R_k R_\varphi^{-1} f]
\]
is real analytic.
\end{Coro}

\begin{proof}
Denote by $\chi(\xi)$ the indicator function of the unit ball in $\R^n$ and by $\chi(D)$ the corresponding Fourier multiplier. We write
\[
 \mathcal R_k = \chi(D) \mathcal R_k + (1-\chi(D)) \mathcal R_k
\]
We consider these two parts seperately. First consider
\[
 R_\varphi (1-\chi(D)) \mathcal R_k R_\varphi^{-1} = R_\varphi (1-\chi(D)) \partial_k \Lambda^{-1} R_\varphi^{-1}
\]
We claim that 
\[
 \varphi \mapsto R_\varphi (\chi(D) + (1-\chi(D)) \Lambda) R_\varphi^{-1}
\]
is real analytic. For the analyticity of $\varphi \mapsto R_\varphi \chi(D) R_\varphi^{-1}$ one can consult \cite{lagrangian}. So one concludes by Lemma \ref{lemma_laplacian} the analyticity of
\[
 \varphi \mapsto R_\varphi \chi(D) R_\varphi^{-1} + R_\varphi (1-\chi(D)) R_\varphi^{-1} R_\varphi \Lambda R_\varphi^{-1}
\]
as a map $\Ds^s(\R^n) \to L(H^s(\R^n);H^{s-1}(\R^n))$. As inversion is an analytic process (see Neumann series) one has also by taking the inverse in $L(H^s(\R^n);H^{s-1}(\R^n))$ that
\[
 \varphi \mapsto R_\varphi \chi(D) R_\varphi^{-1} + R_\varphi (1-\chi(D)) R_\varphi^{-1} R_\varphi \Lambda^{-1} R_\varphi^{-1}
\]
is real analytic. In particular
\[
 \varphi \mapsto R_\varphi (1-\chi(D)) \Lambda^{-1} R_\varphi^{-1}
\]
is real analytic. Further we have
\[
 R_\varphi \partial_k R_\varphi^{-1} f = df [d\varphi]^{-1}
\]
which is a polynomial expression in the first derivatives of $\varphi$ divided by $\det(d\varphi)$ hence analytic in $\varphi$ -- see \cite{lagrangian} for the division by $\det(d\varphi)$. Thus 
\[
 R_\varphi (1-\chi(D)) \mathcal R_k R_\varphi^{-1} = R_\varphi \partial_k R_\varphi^{-1} R_\varphi (1-\chi(D)) \Lambda^{-1} R_\varphi^{-1}
\]
is real analytic in $\varphi$. Now consider the first part of the splitting of $\mathcal R_k$. This is treated in \cite{lagrangian}. There it is shown that expressions of the form
\[
 R_\varphi \chi(D) \mathcal R_k \mathcal R_j R_\varphi^{-1}
\]
are analytic in $\varphi$. In the same manner it follows that
\[
 R_\varphi \chi(D) \mathcal R_k R_\varphi^{-1}
\]
is real analytic in $\varphi$. This concludes the proof.
\end{proof}

Finally we can give the proof of Proposition \ref{prop_analytic}

\begin{proof}[Proof of Proposition \ref{prop_analytic}]
Consider the map
\[
 \varphi \mapsto R_\varphi \mathcal R_k R_\varphi^{-1} f
\]
which by Corollary \ref{coro_riesz} is analytic. So is its derivative. We take the derivative in direction $w \in H^s(\R^n;\R^n)$ and get
\[
 R_\varphi \left(\mathcal R_k d(f \circ \varphi^{-1})) w \circ \varphi^{-1}\right) - R_\varphi \left(\mathcal R_k [d(f \circ \varphi^{-1}) w \circ \varphi^{-1}]\right)  
\]
or using the commutator notation 
\[
 R_\varphi [\mathcal R_k,(w \circ \varphi^{-1}) \cdot \nabla] (f \circ \varphi^{-1})
\]
If we plug in the analytic expression $R_\varphi \mathcal R_j R_\varphi^{-1} g$ for $f$ we see that the expressions appearing in $B(v \circ \varphi^{-1},v \circ \varphi^{-1}) \circ \varphi$ are analytic expressions of $\varphi$ which proves the proposition. 
\end{proof}

\bibliographystyle{plain}

\flushleft
\author{ Hasan Inci\\
EPFL SB MATHAA PDE \\
MA C1 627 (B\^atiment MA)\\ 
Station 8 \\
CH-1015 Lausanne\\
Schwitzerland\\
        {\it email: } {hasan.inci@epfl.ch}
}

\end{document}